\documentclass{amsart}

\usepackage{amssymb, amsmath}
\usepackage{mathrsfs}
\usepackage{amscd}
\usepackage[active]{srcltx}
\usepackage{verbatim}
\usepackage{graphicx}
\usepackage[colorlinks,linkcolor={blue},citecolor={blue},urlcolor={red},]{hyperref}

\usepackage[fleqn,tbtags]{mathtools}
\mathtoolsset{showonlyrefs,showmanualtags}

\usepackage{newsymbol}

\let\mathcal \undefined
\def\mathcal{\mathscr}

\let\emptyset \undefined
\let\ge       \undefined
\let\le       \undefined
\newsymbol\le          1336  
\newsymbol\ge          133E  \let\geq\ge
\newsymbol\emptyset    203F
\newsymbol\notle       230A
\newsymbol\notge       230B

\theoremstyle{plain}
\newtheorem{theorem}{Theorem}[section]
\theoremstyle{remark}
\newtheorem{remark}[theorem]{Remark}

\newtheorem{facts}[theorem]{Facts}
\theoremstyle{plain}
\newtheorem{corollary}[theorem]{Corollary}
\newtheorem{lemma}[theorem]{Lemma}

\newtheorem{definition}[theorem]{Definition}

\numberwithin{equation}{section}

\def\R{{\mathbb R}}
\def\C{{\mathbb C}}

\renewcommand{\a}{\alpha}
\renewcommand{\b}{\beta}

\newcommand{\W}{\mathscr{W}}

\renewcommand{\Re}{\hbox{\rm Re}\,}
\renewcommand{\Im}{\hbox{\rm Im}\,}
\newcommand{\calL}{\mathscr{L}}
\newcommand{\n}{\Vert}
\newcommand{\s}{^*}

\newcommand{\wh}{\widehat}
\newcommand{\ud}{\,{\rm d}}
\newcommand{\dd}{{\rm d}}
\newcommand{\calF}{\mathscr{F}}

\allowdisplaybreaks 

\begin{document}

 \title[Weyl calculus and restricted $L^p$-$L^q$ boundedness]
 {Weyl calculus with respect to the Gaussian measure \\ and restricted 
 $L^p$-$L^q$ boundedness of the \\ Ornstein-Uhlenbeck semigroup in complex time}

\author{Jan van Neerven}

\address{Delft Institute of Applied Mathematics\\
Delft University of Technology\\P.O. Box 5031\\2600 GA Delft\\The Netherlands}
\email{J.M.A.M.vanNeerven@tudelft.nl}

\author{Pierre Portal}

\address{The Australian National University, 
Mathematical Sciences Institute, John Dedman Building, 
Acton ACT 0200, Australia, and Universit\'e Lille 1,
Laboratoire Paul Painlev\'e, F-59655 Villeneuve d'Ascq, France.}
\email{Pierre.Portal@anu.edu.au}     

\date{\today}

 

\thanks{The authors gratefully acknowledge financial support by the ARC discovery Grant DP 160100941}

\begin{abstract} In this paper, we introduce a Weyl functional calculus $a 
\mapsto a(Q,P)$ 
for the position and momentum operators $Q$ and $P$ associated with the 
Ornstein-Uhlenbeck operator
$ L =  -\Delta  + x\cdot \nabla$, and give a simple criterion for restricted 
$L^p$-$L^q$ boundedness of operators in this functional calculus. The analysis of 
this non-commutative functional calculus is simpler than the analysis of the 
functional calculus of $L$. It allows us to recover, unify, and 
extend, old and new results concerning the boundedness of 
$\exp(-zL)$ as an operator from 
$L^p(\R^d,\gamma_{\a})$ to $L^q(\R^d,\gamma_{\b})$f
for suitable values of $z\in \C$ with $\Re z>0$, $p,q\in [1,\infty)$, and
$\a,\b>0$. Here, $\gamma_\tau$ denotes the centred Gaussian measure on $\R^d$
with density $(2\pi\tau)^{-d/2}\exp(-|x|^2/2\tau)$.
\end{abstract}

\maketitle

\section{Introduction}\label{sec:introduction}
In the standard euclidean situation, 
pseudo-differential calculus arises as the Weyl joint functional calculus of a 
non-commuting pair of operators: the position and momentum operators (see, e.g., 
\cite{Sch} and 
\cite[Chapter XII]{Stein}). By 
transferring this calculus to the Gaussian setting,
in this paper we introduce a Gaussian 
version of the Weyl pseudo-differential calculus which assigns to suitable 
functions 
$a:\R^d\times \R^d \to \C$ a bounded operator $a(Q,P)$ acting on $L^2(\R^d,\gamma)$. 
Here,
$Q = (Q_1,\dots,Q_d)$ and $P= (P_1,\dots,P_d)$ are the position  and momentum 
operators 
associated with the Ornstein-Uhlenbeck operator $$L = -\Delta  + x\cdot 
\nabla$$ 
on $L^2(\R^d,\gamma)$, where ${\rm{d}}\gamma(x) = 
(2\pi)^{-d/2}\exp(-\tfrac12|x|^2)\ud x$ is the standard Gaussian measure
on $\R^d$. 
We show that the Ornstein-Uhlenbeck semigroup $\exp(-tL)$ can be expressed 
in terms of this calculus by the formula
\begin{align} \label{timechange} \exp(-tL) = 
\Bigl(1+\frac{1-e^{-t}}{1+e^{-t}}\Bigr)^d\exp\Bigl( -\frac{1-e^{-t}}{1+e^{-t}}(P^2+Q^2) 
\Bigr).\end{align}
With $s:=\frac{1-e^{-t}}{1+e^{-t}}$, the expression on the right-hand side 
is defined through the Weyl calculus as
$(1+s)^d a_s(Q,P)$, where $a_s(x,\xi) = \exp(-s(|x|^2 + |\xi|^2))$. 
The main ingredient in the proof of \eqref{timechange} 
is the explicit determination of the integral 
kernel for $a_s(Q,P)$. By applying a
Schur type estimate to this kernel we are able to prove the following 
sufficient condition for restricted $L^p$-$L^q$-boundedness of $a_s(Q,P)$:

\begin{theorem}
\label{thm:2}
Let $p,q\in [1, \infty)$
and let 
$\a,\b>0$.
For $s\in \C$ with $\Re s >0$, define $r_\pm(s):= \frac12\Re(\frac1s\pm s)$. If 
$s$ satisfies $1-\frac{2}{\a p}+  r_+(s)>0$, $\frac{2}{\b q}-1 + r_+(s)>0$,  
and 
\begin{align*}
(r_-(s))^2 \le  \bigl(1-\frac{2}{\a p} + r_+(s)\bigr)\bigl(\frac{2}{\b q}-1+ 
r_+(s)\bigr), 
\end{align*}
then the operator
$\exp(-s(P^2+Q^2))$ is bounded from $L^p(\R^d,\gamma_{\a})$ to 
$L^{q}(\R^d,\gamma_{\b})$.
\end{theorem}
Here, $\gamma_\tau$ denotes the centred Gaussian measure on $\R^d$
with density $(2\pi\tau)^{-d/2}\exp(-|x|^2/2\tau)$ (so that $\gamma_1 = \gamma$ is the standard Gaussian measure).
The proof of the theorem provides an explicit estimate for the norm of this 
operator 
that is of the correct order in the variable $s$, as subsequent corollaries 
show.

Taken together, \eqref{timechange} and Theorem \ref{thm:2}  can then be used to obtain 
criteria for  $L^p(\R^d,\gamma_{\a})$-$L^{q}(\R^d,\gamma_{\b})$ boundedness of $\exp(-zL)$ for suitable values of
$z\in \C$ with $\Re z>0$.
Among other things, in Section \ref{sec:bbg} we show that the operators 
$\exp(-zL)$ 
map  
$L^{1}(\R^{d}; \gamma_{2})$ to $L^{2}(\R^{d};\gamma)$ for all $\Re z>0$.
We also  prove a more precise
boundedness result which, for real values $t>0$, implies the boundedness of 
$\exp(-tL)$ from 
$L^1(\R^d,\gamma_{\a_t})$ to $L^2(\R^d,\gamma)$, where $\a_t = 1+e^{-2t}$.
The boundedness of these operators was proved recently by Bakry, Bolley, and 
Gentil \cite{bbg} as a corollary of their work on 
hypercontractive bounds on Markov kernels for diffusion semigroups. As such, our 
results may be interpreted as giving an
extension to complex time of the Bakry-Bolley-Gentil result for the 
Ornstein-Uhlenbeck semigroup.

In the final Section \ref{sec:Epp} we show that Theorem \ref{thm:2} also 
captures the well-known result of Epperson \cite{Epp} (see also Weissler \cite{Wei} for
the first boundedness result of this kind, and part of the contractivity result)
for $1<p\le q<\infty$, the operator $\exp(-zL)$ is bounded
from $L^p(\R^d,\gamma)$ to $L^q(\R^d,\gamma)$ if and only if $\omega:= e^{-z}$ 
satisfies
$ |\omega|^2 < p/q$  and 
\begin{equation}\label{eq:Epp2}
 (q-1)|\omega|^4 + (2-p-q) (\Re\omega)^2 - (2-p-q+pq)(\Im\omega)^2 + p-1 > 0.
\end{equation}
In particular, for $p=q$ the semigroup $\exp(-tL)$ on $L^p(\R^d,\gamma)$
extends analytically to the set (see Figure 1) 
\begin{align}\label{eq:Epp}E_p := \{z = x+iy\in\C: \, |\sin(y)| < 
\tan(\theta_p)\sinh(x)\},
\end{align}
where 
\begin{align}\label{eq:theta-p}
\cos \phi_p = \Big|\frac2p -1\Big|.
\end{align}
 
 \begin{center}
 \begin{figure}[ht]\label{fig:Epp}
  \includegraphics[scale=0.5]{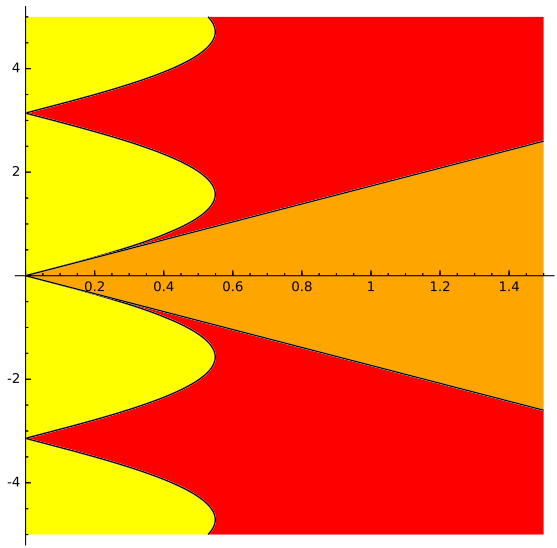}
 
 \caption{The Epperson region $E_p$ (red/orange) and the sector with angle 
 $\theta_p$ for $p = 4/3$ (orange).}
 \end{figure}
 \end{center}


These results demonstrate the potential of the Gaussian 
pseudo-differential calculus. Of course, taking \eqref{timechange} for granted, 
we could forget about the Gaussian pseudo-differential calculus 
altogether, and reinterpret all the applications given in this paper as 
consequences of the realisation that
through the time change $s \mapsto 
\frac{1-e^{-t}}{1+e^{-t}}$, various algebraic simplifications allow one to 
derive sharp results for the Ornstein-Uhlenbeck semigroup in a unified 
manner. 
In fact, as the referee of this paper pointed out to us, Weissler took exactly this approach in \cite{Wei}, and obtained the most important special case of our Theorem \ref{thm:2} in 1979. Besides generalising this result to the context of weighted Gaussian measures $\gamma_{\alpha}$ arising from \cite{bbg}, the point of the new approach given here is to connect results such as Weissler's, and other classical hypercontractivity theorems, to the underlying Weyl calculus. In doing so, one sees the 
reason why certain crucial algebraic simplifications occur, and one develops a far more flexible tool to study other 
spectral 
multipliers associated with the Ornstein-Uhlenbeck operator (and, possibly, 
perturbations 
thereof). In such applications, the algebraic consequences of the fact that the 
Weyl calculus involves non-commuting operators may not be as easily unpacked as 
in 
\eqref{timechange}. The $L^p$-analysis of operators in the Weyl calculus of the 
pair $(Q,P)$, 
however, is simpler than the direct analysis of operators in the functional 
calculus of $L$ (or 
perturbations of $L$). In future works, we plan to develop this theory and 
include harmonic analysis substantially more advanced than the Schur type 
estimate employed here, along 
with applications to non-linear stochastic differential equations.

\section*{Acknowledgements} 
We are grateful to the anonymous referee for her/his useful suggestions, and, in particular, for pointing out to us the paper \cite{Wei}.

\section{The Weyl calculus with respect to the Gaussian measure}\label{sec:WC}

In this section we introduce the Weyl calculus with respect 
to the Gaussian measure.
To emphasise its Fourier analytic content, our point of departure is the fact
that Fourier-Plancherel transform is unitarily equivalent to 
the second quantisation of multiplication by $-i$.
The unitary operator implementing this equivalence is used 
to define the position and momentum operators $Q$ and $P$ associated with the 
Ornstein-Uhlenbeck operator $L$. This approach bypasses the use of creation
and annihilation operators altogether and leads to the same expressions.

\subsection{The Wiener-Plancherel transform with respect to the Gaussian 
measure}\label{subsec:W}
Let $\!\ud m(x) = (2\pi)^{-d/2}\ud x$ denote the normalised Lebesgue measure on $\R^d$. 
The mapping $E: f\mapsto ef$, where 
$$e(x) := 
\exp(-\tfrac14|x|^2),$$ is unitary from 
$L^2(\R^d,\gamma)$ onto $L^2(\R^d,m)$,
and the dilation $\delta: L^2(\R^d,m)\to L^2(\R^d,m)$,
$$ \delta f(x) := (\sqrt 2)^{d} f\bigl({\sqrt 2}x\bigr)$$
is unitary on $L^2(\R^d,m)$.
Consequently the operator
$$ U:= \delta \circ E$$
is unitary from $L^2(\R^d,\gamma)$ onto $L^2(\R^d,m)$.
It was shown by Segal \cite[Theorem 2]{Seg} that $U$ establishes a unitary 
equivalence 
\begin{align}\label{eq:FW} 
\W  = U^{-1} \circ \calF \circ U
\end{align}
of the Fourier-Plancherel transform $\calF$ as a unitary operator on $L^2(\R^d,m)$,
\begin{align}\label{eq:FT} \calF f(y) := \wh f(y): =  \frac1{(2\pi)^{d/2}}\int_{\R^d} f(x) 
\exp(-ix\cdot y)\ud x  = \int_{\R^d} f(x) 
\exp(-ix\cdot y)\ud m(x),
\end{align}
with 
the unitary operator $\mathscr{W}$ on $L^2(\R^d,\gamma)$,  
defined for polynomials $f$ by
$$ \W f(y):= \int_{\R^d} f(-iy+\sqrt 2 x)\ud \gamma(x).
 $$
We have the following beautiful representation of this operator, which is
sometimes called the {\em Wiener-Plancherel transform}, in terms of the second 
quantisation functor $\Gamma$ \cite[Cor\-ollary 3.2]{Seg}: 
 $$\W = \Gamma(-i).$$
This identity is not used in the sequel, but it is stated only to demonstrate 
that both the operator $\W$ and the unitary $U$ are very natural. 

\subsection{Position and momentum with respect to the Gaussian 
measure}\label{subsec:PQ}

Consider classical position and momentum operators 
$$X=(x_1,\dots,x_d), \quad D = 
(\frac1i\partial_1,\dots,\frac1i\partial_d),$$
viewed as densely defined operators mapping from $L^2(\R^d)$ into $L^2(\R^d;\C^d)$. Explicitly, 
 $x_j$ is the densely defined self-adjoint operator  
on $L^2(\R^d)$ defined by pointwise multiplication, i.e., $(x_j f)(x):= x_j f(x)$ for $x\in \R^d$, with maximal domain $$\mathsf{D}(x_j) = \{f\in L^2(\R^d):\, x_j f\in L^2(\R^d)\},$$
and $\frac1i\partial_j$ is the self-adjoint operator $f\mapsto \frac1i\partial_h f$
with maximal domain $$\mathsf{D}(\frac1i\partial_j) = \{f\in L^2(\R^d):\, \partial_j f\in L^2(\R^d)\},$$
the partial derivative being interpreted in the sense of distributions.

Having motivated our choice of the unitary $U$, we now use it to introduce 
the position and momentum operators 
$Q=(q_1,\dots,q_d)$ and $P=(p_1,\dots,p_d)$
as densely defined closed operators acting from their natural domains in 
$L^2(\R^d,\gamma)$ into $L^2(\R^d,\gamma;\C^d)$ by unitary equivalence
with $X=(x_1,\dots,x_d)$ and $ D = 
(\frac1i\partial_1,\dots,\frac1i\partial_d)$:
\begin{align}\label{reason:3} 
 q_j & :=  U^{-1}\circ x_j \circ U,\\
 p_j & := U^{-1}\circ \frac1{i} \partial_j\circ U.
\end{align}
They satisfy the commutation relations 
\begin{equation}\label{reason:1} 
\begin{aligned}
 [p_j,p_k] = [q_j,q_k] & = 0, \quad   
 [q_j,p_k] = \frac1{i} \delta_{jk},
\end{aligned}
\end{equation}
as well as the identity
\begin{equation}\label{eq:L}
 \frac12(P^2+Q^2) = L +\frac{d}{2}I. 
\end{equation}
Here, 
$L$ is the {\em Ornstein-Uhlenbeck operator} 
which acts on test functions $f\in C_{\rm c}^2(\R^d)$  
by $$L f(x) := -\Delta f(x) + x\cdot \nabla f(x)\quad  (x \in \R^{d}).$$
It follows readily from the definition of the Wiener-Plancherel transform $\W$ 
that  
\begin{align*} q_j\circ \W 
= \W \circ p_j,\\
  p_j\circ \W 
  = -\W\circ  q_j
\end{align*}
consistent with the relations $x_j\circ  \calF = \calF\circ (\frac1i 
\partial_j)$
and  $(\frac1i \partial_j)\circ \calF = -\calF \circ x_j$ for position and 
momentum in the Euclidean setting.

\begin{remark} Our definitions of $P$ and $Q$ coincide 
with the physicist's definitions
in the theory of the quantum harmonic oscillator  
(cf. \cite{GK}). Other texts, such as \cite{Par}, use different 
normalisations. The present choice
makes the commutation relation between position and momentum as well as 
the identity relating the Ornstein-Uhlenbeck operator
and position and momentum come out right in the sense that \eqref{reason:1} and 
\eqref{eq:L} hold.
The former says that position and momentum satisfy the `canonical commutation 
relations'
and the latter says that the Hamiltonian $\frac12(P^2+Q^2)$ of 
the quantum harmonic oscillator  
equals the number operator $L$ (physicists would write $N$) plus the ground 
state energy 
$\frac{d}{2}$.
\end{remark}

\subsection{The Weyl calculus with respect to the Gaussian 
measure}\label{subsec:Weyl}

The Weyl calculus for the pair $(X,D)$ is defined, for 
Schwartz functions $a: \R^{2d}\to\C$, by  
\begin{equation}\label{eq:a1}
a(X,D)f(y)
    = \int_{\R^{2d}} \wh a(u,v) \exp(i(uX+vD))
f(y) \ud m(u)\ud m(v).
\end{equation}
Here $m(\dd x) = (2\pi)^{-d/2}\ud x$ as before, 
$\wh a := \calF a$ is the Fourier-Plancherel transform of $a$, 
and the unitary operators  $\exp(i(uX+v D))$ on $L^2(\R^d,\gamma)$ 
are defined through the action 
\begin{align}\label{eq:exppq} \exp(i(uX+v D))
f(y) := \exp(iuy + \tfrac12 iuv)f(v+y) 
\end{align}
(cf. \cite[Formula 51, page 550]{Stein}). This definition can be motivated by a 
formal application of the 
Baker-Campbell-Hausdorff formula to the (unbounded) operators $X$ and $D$; 
alternatively, one may look upon it as defining 
a unitary representation of the Heisenberg group 
encoding the commutation relations of $X$ and $D$, the 
so-called Schr\"odinger representation. 

Motivated by the constructions in the preceding subsection, we make the following definition.

\begin{definition}
 For  $u,v \in \R^d$,
on $L^2(\R^d,\gamma)$  we define the unitary operators 
 $\exp(i(uQ+v P))$ on $L^2(\R^d;\gamma)$ by
 $$ \exp(i(uQ+v P)) := U^{-1}\circ \exp(i(uX+v D))\circ U.$$
\end{definition}
This allows us to define, for Schwartz functions $a: \R^{2d}\to\C$, 
the bounded operator $a(Q,P)$ on $L^2(\R^d,\gamma)$ by
\begin{align}\label{eq:aPQ} a(Q,P) = U^{-1}\circ a(X,D)\circ U = \int_{\R^{2d}} \wh a(u,v) \exp(i(uQ+vP))
\ud m(u)\ud m(v),
\end{align}
the integral being understood in the strong sense.
An explicit expression for $a(Q,P)$ can be obtained as follows.
By \eqref{eq:exppq} and a change of variables one has
(cf. \cite[Formula (52), page 551]{Stein})
\begin{equation}\label{eq:a2}\begin{aligned}
a(X,D)f(y)
= \int_{\R^{2d}} 
a(\tfrac12(v+y),\xi)\exp(-i\xi(v-y))f(v)\ud m(v)\ud m(\xi).
\end{aligned}
\end{equation}
By \eqref{eq:aPQ} and the definition of $U$, this
gives the following explicit formula for the Gaussian setting:
\begin{equation}\label{eq:a-gauss}
\begin{aligned}
  a(Q,P)f(y) 
& =\frac1{(2\pi)^d}\int_{\R^{2d}}  a(\tfrac12(x+\frac{y}{\sqrt 
2}),\xi)\exp(-i\xi(x-\frac{y}{\sqrt{2}}
))\exp(-\tfrac12|x|^2+\tfrac14|y|^2)f(x\sqrt{2})\ud \xi\ud x
\\ & =\frac{1}{(2\sqrt 2\pi)^d}\int_{\R^{2d}}  a(\frac{x+y}{2\sqrt 
2},\xi)\exp(-i\xi(\frac{x-y}{\sqrt{2}}
))\exp(-\tfrac14|x|^2+\tfrac14|y|^2)f(x)\ud \xi\ud x
\\ & = \int _{\R^{d}} K_{a}(y,x)f(x)\ud x,
\end{aligned}
\end{equation}
where 
\begin{align}\label{eq:Ka} K_{a}(y,x) := \frac{1}{(2\sqrt 2\pi)^d}\exp(-\tfrac14|x|^2+\tfrac14|y|^2)\int_{\R^{d}}  
a(\frac{x+y}{2\sqrt 
2},\xi)\exp(-i\xi(\frac{x-y}{\sqrt{2}}))\ud 
\xi.
\end{align}

\section{Expressing the Ornstein-Uhlenbeck semigroup in the Weyl 
calculus}\label{sec:OU}

In order to translate results about the Weyl 
functional calculus of $(Q,P)$ into results regarding the functional calculus 
of 
$L$, we first need to relate these two calculi. This is done in the next 
theorem. It is the only place where we rely on the concrete expression of 
the Mehler kernel.

\begin{theorem}\label{thm:exptL} For all $t > 0$ we have, 
with $s:= \frac{1-e^{-t}}{1+e^{-t}}$,  
\begin{equation}\label{eq:OU-via-PQ}
 \exp(-tL) = 
(1+s)^d\,  a_{s}(Q,P),
\end{equation}
where $ a_s(x,\xi):= \exp(-s(|x|^2+|\xi|^2)).$ 
\end{theorem}
In the next section we provide restricted $L^p$-$L^q$ estimates for $a_s(Q,P)$ 
for complex values of $s$ purely based on the Weyl calculus. 

We need an elementary calculus lemma 
which is proved by writing out the inner product and square norm in terms of 
coordinates, thus  
writing the integral as a product of $d$ integrals with respect to a single 
variable.

\begin{lemma}\label{lem:calc} For all $A>0$, $B\in\R$, and $y\in\R^d$,  
$$ \int_{\R^d} \exp(-A|y|^2 +Bxy)\ud x  = \bigl(\frac{\pi}{A}\bigr)^{d/2} 
\exp\bigl(\frac{B^2}{4A}|y|^2\bigr). 
$$ 
\end{lemma}

\begin{proof}[Proof of Theorem \ref{thm:exptL}]
By \eqref{eq:Ka} we have
\begin{equation}\label{eq:Ka2}
\begin{aligned} 
K_{a_s}(y,x) & = \frac{1}{(2\sqrt 2\pi)^d}\exp(-\tfrac14|x|^2+\tfrac14|y|^2)
\int_{\R^{d}} 
\exp(-s(|\xi|^2+\tfrac18|x+y|^2))\exp(-i\xi(\frac{x-y}{\sqrt{2}}))\ud \xi 
\\ & =  \frac{1}{(2\sqrt 2\pi)^d}\exp(-\tfrac{s}{8}|x+y|^2)\exp(-\tfrac14|x|^2+\tfrac14|y|^2)
\int_{\R^d} \exp(-s(|\xi|^2 
+\tfrac{i}{s}\xi(\frac{x-y}{\sqrt 2}))) \ud \xi
\\ & =  \frac{1}{(2\sqrt 2\pi)^d}\exp(-\tfrac{1}{8s}|x-y|^2)\exp(-\tfrac{s}{8}
|x+y|^2)\exp(-\tfrac14|x|^2+\tfrac14|y|^2)
\int_{\R^d} \exp(-s|\eta|^2) \ud \eta
\\ & = \frac1{2^d(2\pi s)^{d/2}} \exp(-\tfrac{1}{8s}|x-y|^2) 
\exp(-\tfrac{s}{8}|x+y|^2)\exp(-\tfrac14|x|^2+\tfrac14|y|^2)
\\ & = \frac1{2^d(2\pi 
s)^{d/2}}\exp(-\tfrac{1}{8s}(1-s)^2(|x|^2+|y|^2)+\tfrac{1}{4}(\tfrac1s - 
s)xy)\exp(-\tfrac12|x|^2)
\end{aligned}
\end{equation}
and therefore 
\begin{equation}\label{eq:kernel}
\begin{aligned}
\exp(-s(P^2 + Q^2))f(y) &  = \int _{\R^{d}} K_{a_s}(y,x)f(x)\ud x
\\ & = \frac1{2^d (2\pi s)^{d/2}}
\int_{\R^d} \exp(-\tfrac{1}{8s}(1-s)^2(|x|^2+|y|^2)+\tfrac{1}{4}(\tfrac1s - s)xy) f(x)
e^{-1/2|x|^2}\ud x.
\end{aligned}
\end{equation}
Taking $s:= \frac{1-e^{-t}}{1+e^{-t}}$ in this identity we obtain
\begin{equation}\label{eq:mehler}
\begin{aligned}
 \ & \bigl(1+\frac{1-e^{-t}}{1+e^{-t}}\Bigr)^d\exp\Bigl( 
-\frac{1-e^{-t}}{1+e^{-t}}(P^2+Q^2) \Bigr)f(y)
\\ & \qquad = \frac1{(2\pi)^{d/2}} \Bigl(\frac{2}{1+e^{-t}}\Bigr)^d\frac1{2^d} 
\Bigl(\frac{1+e^{-t}}{1-e^{-t}}\Bigr)^{d/2} \\ & \qquad\qquad \times 
\int_{\R^d} 
\exp\Bigl(-\frac{1}{2}\frac{e^{-2t}}{1-e^{-2t}}(|x|^2+|y|^2)+{\frac{e^{-t}}{1-e^
{-2t}}}xy\Bigr)
f(x)\exp(-\tfrac12|x|^2)\ud x
\\ & \qquad = \frac1{(2\pi)^{d/2}}  \Bigl(\frac{1}{1-e^{-2t}}\Bigr)^{d/2}
\int_{\R^d} \exp\Bigl(-\frac12\dfrac{|e^{-t}y - x|^2}{1
      - e^{-2 t}}  \Bigr)f(x)\ud x
\\ & \qquad = \int_{\R^d} M_t(y,x)f(x)\ud x 
      \\ & \qquad = \exp(-tL)f(y),
\end{aligned}
\end{equation}
 where 
$$M_t(y,x) = \frac1{(2\pi)^{d/2}}  \Bigl(\frac{1}{1-e^{-2t}}\Bigr)^{d/2}
\exp\Bigl(-\frac12\dfrac{|e^{-t}y - x|^2}{1- e^{-2 t}}  \Bigr)
$$
denotes the Mehler kernel; 
the last step of \eqref{eq:mehler} uses the classical Mehler formula for $\exp(-tL)$. 
\end{proof}

For any $z\in \C$ with $\Re z>0$, the operator $\exp(-zL)$ is well defined and 
bounded as a 
linear operator on $L^2(\R^d,\gamma)$, and the same is true for the 
expression on the right-hand side in \eqref{eq:OU-via-PQ} by 
analytically extending the kernel defining it.
By uniqueness of analytic extensions, the identity \eqref{eq:OU-via-PQ} 
persists for complex time.

The identity \eqref{eq:OU-via-PQ}, extended analytically into the complex plane, 
admits the following deeper interpretation. 
The transformation 
\begin{equation}\label{eq:bihol} s =  \frac{1-e^{-z}}{1+e^{-z}},
\end{equation}
which is implicit in Theorem \ref{thm:exptL}, 
is bi-holomorphic from $$\{z\in \C:\,\Re z>0,\,|\Im(z)|< \pi\}$$
onto $$\{s\in \C: \ \Re s > 0, \ s\not\in [1,\infty)\}.$$ 
For $1<p<\infty$ it maps $E_p \cap\{z\in \C:\,|\Im(z)|< \pi\}$, where $E_p$ is the Epperson region
defined by \eqref{eq:Epp}, 
onto $\Sigma_{\theta_p}\setminus [1,\infty)$, where 
$\Sigma_{\theta_p} = \{s\in \C: \, s\not=0, |\arg(s)| < \theta_p\}$ is the open
sector with angle $\theta_p$ given by \eqref{eq:theta-p} 
(see Figure 1).
Using the periodicity modulo $2\pi i$ of the exponential function, the mapping \eqref{eq:bihol} 
maps $E_p$ onto $\Sigma_{\theta_p}\setminus\{1\}$.

Using this information, the analytic extendibility of the semigroup $\exp(-tL)$ on $L^p(\R^d,\gamma)$ 
to $E_p$ can now be proved by showing that that $\exp(-s(P^2+Q^2))$ extends 
analytically to $\Sigma_{\theta_p}$; the details are presented in Theorem \ref{thm:Epp}.
This shows that $\exp(-s(P^2+Q^2))$ is a much simpler object than 
$\exp(-zL)$. 

\begin{remark}
By \eqref{eq:a-gauss} and \eqref{eq:mehler}, the theorem can be interpreted as giving 
a representation of the Mehler kernel in terms of the variable $\frac{1-e^{-t}}{1+e^{-t}}$.
This representation could be taken as the starting point for the 
results in the next section without any reference to the Weyl calculus. 
As we already pointed out in the Introduction,
this would obscure the point that the Weyl calculus explains why 
the ensuing algebraic simplifications occur. What is more, the calculus 
 can be applied to other functions $a(x,\xi)$ beyond 
the special choice $a_s(\xi,x) = \exp(-s(|x|^2+|\xi|^2))$ and may serve as 
a tool to study spectral multipliers associated with the Ornstein-Uhlenbeck operator.
\end{remark}

\section{Restricted $L^p$-$L^q$ estimates for $\exp(-s(P^2+Q^2))$ }

Restricting the operators $\exp(-s(P^2+Q^2))$ to $C_{\rm c}^{\infty}(\R^{d})$,
we now take up the problem of determining when these restrictions extend to bounded operators 
from 
$L^p(\R^d,\gamma_{\a})$ into $L^q(\R^d,\gamma_{\b})$. Here, for $\tau>0$, we set
$${\rm d}\gamma_\tau(x) = (2\pi\tau)^{-d/2}\exp(-|x|^2/2\tau)\,{\rm d}x$$ 
(so that $\gamma_1 =\gamma$ is the standard Gaussian measure). Boundedness (or 
rather, contractivity)
from $L^p(\R^d,\gamma)$ to $L^q(\R^d,\gamma)$ corresponds to classical 
hypercontractivity of the Ornstein-Uhlenbeck semigroup.
For other values of $\a,\b>0$ this includes restricted ultracontractivity of the 
kind obtained in \cite{bbg}. 

We begin with a sufficient condition for 
$L^p(\R^d,\gamma_{\a})$-$L^q(\R^d,\gamma_{\b})$ boundedness (Theorem 
\ref{thm:main} below). 
Recalling that $\exp(-s(P^2+Q^2))$ equals the integral operator with kernel 
$K_{a_s}$ given by \eqref{eq:Ka2}, 
an immediate sufficient condition for boundedness derives from 
H\"older's inequality: if $p,q\in [1,\infty)$ and 
$\frac1p+\frac1{p'} =1$, and  
\begin{equation}\label{eq:Holder}
\int_{\R^{d}} \Bigl( \int_{\R^{d}} |K_a(y,x)|^{p'} \exp\bigl(\frac{p'}{2\a 
p}|x|^{2}\bigr)\ud x \Bigr)^{q/p'} \exp(-|y|^{2}/2\beta)\ud y =:C <\infty
\end{equation}
(with the obvious change if $p=1$)
then $a(Q,P)$ extends to a bounded operator from $L^{p}(\R^d,\gamma_{\a})$ to 
$L^{q}(\R^d,\gamma_{\b})$
with norm at most $C$. 
A much sharper criterion can be given by using the following Schur type estimate 
(which is 
a straightforward refinement of
\cite[Theorem 0.3.1]{Sogge}).

\begin{lemma}\label{lem:Schur}
Let $p,q,r \in [1,\infty)$ be such that 
$\frac{1}{r}=1-(\frac{1}{p}-\frac{1}{q})$. If 
$K\in L^1_{\rm loc}(\R^2)$ and $\phi,\psi: \R\to (0,\infty)$ 
are integrable functions 
such that  
$$ 
\sup_{y \in\R^d} \Bigl(\int_{\R^d}  
|K(y,x)|^{r}\frac{\psi^{r/q}(y)}{\phi^{r/p}(x)}\ud x\Bigr)^{{1}/{r}} 
=: C_1 <\infty,
$$
and 
$$
\sup_{x\in\R^d} 
\Bigl(\int_{\R^d}  |K(y,x)|^{r} \frac{\psi^{r/q}(y)}{\phi^{r/p}(x)}\ud 
y\Bigr)^{{1}/{r}} =:C_2<\infty
$$ 
then $$T_K f(y):= \int_\R K(y,x)f(x)\ud x \quad (f\in C_{\rm c}(\R))$$ 
defines a bounded operator $T_K$ from $L^{p}(\R^d,\phi(x)\ud x)$ to 
$L^{q}(\R^d,\psi(x)\ud x)$ with norm 
$$ \n T_K\n_{L^{p}(\R^d,\phi(x)\ud x),L^{q}(\R^d,\psi(x)\ud x)} \le 
C_1^{1-\frac{r}{q}} C_2^\frac{r}{q}.$$
\end{lemma}
\begin{proof}
For strictly positive functions $\eta\in L^1(\R)$ denote by $L_\eta^s(\R)$ the 
Banach space of measurable functions $g$ such that $\eta g\in L^s(\R)$,
identifying two such functions $g$ if they are equal almost everywhere.
From 
\begin{align*}| T_K f(y)|  
& \le \int_{\R^d}  |K(y,x)||f(x)| \frac1{\phi^{1/p}(x)} \phi^{1/p}(x)\ud x
\\ & \le \Bigl(\int_{\R^d}  |K(y,x)|^r \frac1{\phi^{r/p}(x)} \ud x\Bigr)^{1/r}\n 
f\n_{L_{\phi^{1/p}}^{r'}(\R)}
\\ & =\frac1{\psi^{1/q}(y)}\Bigl(\int_{\R^d}  |K(y,x)|^r 
\frac{\psi^{r/q}(y)}{\phi^{r/p}(x)} \ud x\Bigr)^{1/r}\n 
f\n_{L_{\phi^{1/p}}^{r'}(\R)}
\end{align*}
we find that 
$$
\n T_K f\n_{L_{\psi^{1/q}}^\infty(\R)}  
\le C_1\n f\n_{L_{\phi^{1/p}}^{r'}(\R)}.
$$
This means that $$T_K:L_{\phi^{1/p}}^{r'}(\R)\to L_{\psi^{1/q}}^\infty(\R)$$
is bounded with norm at most $C_1$.
With $K'(y,x) := \overline{K(x,y)}$, the same argument gives that 
$T_K\s = T_{K'}$ extends to a bounded operator from 
$L_{(1/\psi)^{1/q}}^{r'}(\R)$ to $L_{(1/\phi)^{1/p}}^\infty(\R)$ with norm at 
most $C_2$.
Dualising, this implies that $$T_K: L^1_{\phi^{1/p}}(\R)\to 
L^r_{\psi^{1/q}}(\R)$$ is bounded with norm at most $C_2$.

This puts us into a position to apply the Riesz-Thorin theorem. Choose 
$0<\theta<1$ in such a way that 
$ \frac{1}{p} = \frac{1-\theta}{r'} + \frac{\theta}{1}$, that is, 
$\frac{\theta}{r} = \frac1p -(1 -\frac1r) = \frac1q$, 
so $\theta = \frac{r}{q}$. In view of $\frac{1}{q} = \frac{1-\theta}{\infty} + 
\frac{\theta}{r}$
it follows that $$T_K:L^p_{\phi^{1/p}}(\R)\to L_{\psi^{1/q}}^q(\R)$$ is bounded 
with norm at most 
$C = C_1^{1-\theta}C_2^\theta = C_1^{1-\frac{r}{q}} C_2^\frac{r}{q}$.
But this means that $$T_K: L^p(\R,\phi(x)\ud x)\to L^q(\R,\psi(x)\ud x)$$ is 
bounded with norm at most $C$.
\end{proof}

Motivated by \eqref{eq:kernel}, for $s\in \C$ with $\Re s >0$ we define 
\begin{align}\label{eq:r} b_s:= \tfrac{1}{8s}(1-s)^2, \quad 
c_s:=\tfrac{1}{4}(\tfrac1s - s).
\end{align}
Setting $$ r_\pm(s):= \frac12\Re(\frac1s\pm s)$$ we have the identities
$ \frac14 + \Re b_s  = \frac14 r_+(s)$ and $\Re c_s = \frac12 r_-(s).$

\begin{theorem}[Restricted $L^p$-$L^q$ boundedness]\label{thm:main}
Let $p,q\in [1, \infty)$, let $\frac{1}{r}=1-(\frac{1}{p}-\frac{1}{q})$, and let 
$\a,\b>0$.
If $s\in \C$ with $\Re s >0$ satisfies $1-\frac{2}{\a p}+ r_+(s)>0$, 
$\frac{2}{\b q}-1 + r_+(s)>0$,  and 
\begin{align}\label{eq:s}
(r_-(s))^2 \le  \bigl(1-\frac{2}{\a p} +  r_+(s)\bigr)\bigl(\frac{2}{\b q}-1+ 
r_+(s)\bigr), 
\end{align}
then the operator
$\exp(-s(P^2+Q^2))$ is bounded from $L^p(\R^d,\gamma_{\a})$ to 
$L^{q}(\R^d,\gamma_{\b})$
with norm 
\begin{align*} \n 
\exp(-s(P^2+Q^2))\n_{\calL(L^p(\R^d,\gamma_{\a}),L^{q}(\R^d,\gamma_{\b}))}
\le  \frac1{(2rs)^{d/2}} \frac{
(\frac{\alpha r}{2})^{d/2p}(\frac{\beta r}{2})^{-d/2q}}
{(1 - \tfrac{2}{\a p}+  r_+(s))^{\frac{d}{2}(1-\frac1p)}(\tfrac{2}{\b q}-1+ 
r_+(s))^{\frac{d}{2}\frac1q}}.
\end{align*}
\end{theorem}

\begin{remark}\label{rem:C}
 We have no reason to believe that the numerical constant 
 $(\frac1{2r})^{d/2}(\frac{\alpha r}{2})^{d/2p}(\frac{\beta r}{2})^{-d/2q}$ is 
sharp, but 
the examples that we are about to work out indicate that the dependence on $s$ 
is of the correct order. 
\end{remark}

\begin{remark}\label{rem:pos} For $s = x+iy\in \C$ with $x>0$ we have $r_+(s) 
= \frac12( \frac{x}{x^2+y^2}+x) >0.$
It follows that the positivity assumptions $1-\frac{2}{\a p}+ r_+(s)>0$ and 
$\frac{2}{\b q}-1+ r_+(s)>0$ are fulfilled for all $\Re s>0$ if, respectively, 
$\a p\ge 2$ and $\b q\le 2$.
\end{remark}

\begin{proof}
Using the notation of \eqref{eq:r}, the condition \eqref{eq:s}
is equivalent to 
\begin{align}\label{eq:s-equiv}(\Re c_s)^2 \le 4\bigl(\frac{1}{2}-\frac{1}{2\a 
p}+ \Re b_s\bigr)\bigl(\frac{1}{2\b q}+\Re b_s\bigr).
\end{align}

We prove the theorem by checking the criterion of Lemma \ref{lem:Schur} 
for $K = K_{a_s}$ with $a_s(x,\xi)= \exp(-s(|x|^{2}+|\xi|^{2}))$, and $\phi(x) = 
(2\pi\alpha)^{-d/2}\exp(-|x|^2/2\a)$, $\psi(x) = 
(2\pi\beta)^{-d/2}\exp(-|x|^2/2\b)$. 

By \eqref{eq:Ka2}, for almost all $x,y \in \R^{d}$ we have 
\begin{align*}
K_{a_s}(y,x) &
= \frac1{2^d(2\pi s)^{d/2}}\exp(-b_s(|x|^2+|y|^2)+c_sxy)\exp(-\tfrac12|x|^2).
\end{align*}
Let $r \in [1,\infty)$ be such that $\frac{1}{r}=1-(\frac{1}{p}-\frac{1}{q})$.
Using Lemma \ref{lem:calc}, applied with $A= r( \tfrac{1}{2} - \tfrac{1}{2\a 
p}+\Re b_s)$ and $B=  r\Re c_s$, 
we may estimate 
\begin{equation}\label{estimate-ultra-schur1a}
\begin{aligned}
\ & \sup_{y \in \R^{d}} \Bigl(\int_{\R^{d}} |K_{a_s}(y,x)|^{r}  
(2\pi\a)^{rd/2p}\exp(\tfrac{1}{p}r 
|x|^{2}/2\a)(2\pi\a)^{-rd/2p}\exp(-\tfrac{1}{q}r|y|^{2}/2\b) \ud x\Bigr)^{1/r}
\\ &  = \frac{(2\pi\a)^{d/2p}(2\pi\b)^{-d/2q}}{2^d(2\pi s)^{d/2}} 
\\ &  \qquad\qquad \times
\sup_{y \in \R^{d}} \Bigl( \int_{\R^{d}} \exp(-r\Re b_s(|x|^2+|y|^2)+r\Re 
c_sxy)\exp(-r(\tfrac12-\tfrac{1}{2\a p})|x|^2) \exp(-\tfrac{1}{2\b 
q}r|y|^{2})\ud x \Bigr)^{1/r}
\\ & = \frac{(2\pi\a)^{d/2p}(2\pi\b)^{-d/2q}}{2^d(2\pi s)^{d/2}} 
\\ &  \qquad\qquad\times
\sup_{y \in \R^{d}}\Bigl[ \exp(-(\tfrac{1}{2\b q}+\Re b_s) |y|^2)
\Bigl( \int_{\R^{d}} 
\exp(-r( \tfrac{1}{2} - \tfrac{1}{2\a p}+\Re b_s ) |x|^2+r\Re c_s xy)\ud x 
\Bigr)^{1/r}\Bigr]
\\ &  = \frac{(2\pi\a)^{d/2p}(2\pi\b)^{-d/2q}}{2^d(2 \pi s)^{d/2}} 
\Bigl(\frac{\pi}{r  ( \tfrac{1}{2} - \tfrac{1}{2\a p}+\Re b_s)  }\Bigr)^{d/2r} 
\\ &  \qquad\qquad \times
\sup_{y \in \R^{d}} \Bigl[\exp(-(\tfrac{1}{2\b q}+\Re b_s) |y|^2)
\exp\Bigl(\frac{(\Re c_s)^{2}}{4( \frac{1}{2} - \tfrac{1}{2\a p}+\Re b_s 
)}|y|^2\Bigr)\Bigr]
\\ &  =  \frac{(2\pi\a)^{d/2p}(2\pi\b)^{-d/2q}}{2^d(2 
\pi)^{d/2}}\Bigl(\frac{\pi}{r}\Bigr)^{d/2r} \frac1{s^{d/2}}  \Bigl(\frac{1}{ 
\tfrac{1}{2} - \tfrac{1}{2\a p}+\Re b_s  }\Bigr)^{d/2r} 
\\ &  = \frac{(2\a)^{d/2p}(2\b)^{-d/2q}}{2^{{3d}/{2}}}\frac{1}{r^{d/2r}} 
\frac1{s^{d/2}}  \Bigl(\frac{1}{ \tfrac{1}{2} - \tfrac{1}{2\a p}+\Re b_s  
}\Bigr)^{d/2r}.
\end{aligned}
\end{equation}

In the same way, using 
Lemma \ref{lem:calc} applied with $A= r(\frac{1}{\b q}+\Re b_s)$ and $B=  r\Re 
c_s$,
\begin{align}\label{estimate-ultra-schur1b}
\sup_{x \in \R^{d}} \Bigl(\int_{\R^{d}} |K_{a_s}(y,x)|^{r}  
\exp(\tfrac{1}{2\a p}r  |x|^{2})\exp(-\tfrac{1}{2\b q}r|y|^{2})\ud y\Bigr)^{1/r}
\!= 
\frac{(2\a)^{d/2p}(2\b)^{-d/2q}}{2^{{3d}/{2}}}\frac{1}{r^{d/2r}}\frac1{s^{d/2}} 
\Bigl(\frac{1}{ \tfrac{1}{2\b q}+\Re b_s }\Bigr)^{d/2r}\!.
\end{align}

Denoting these two bounds by $C_1$ and $C_2$,
Lemma \ref{lem:Schur} bounds the 
norm of the operator
by $C_1^{1-\frac{r}{q}}C_2^{\frac{r}{q}} = 
C_1^{r(1-\frac{1}{p})}C_2^{\frac{r}{q}}$.
After rearranging the various constants a bit, this gives the estimate in the 
statement of the theorem.
\end{proof}

\begin{remark}
In the above proof one could replace the Schur test (Lemma \ref{lem:Schur}) by the  weaker 
condition \eqref{eq:Holder} based on H\"older's inequality. 
This would have the effect of replacing the suprema by integrals throughout the proof. 
This leads not only to sub-optimal estimates, but more importantly it would not allow 
to handle the critical case when \eqref{eq:s} holds with equality. 
\end{remark}

Combining Theorems \ref{thm:exptL} and \ref{thm:main}, we obtain the following 
boundedness result
for the operators $\exp(-zL)$. 

\begin{corollary}
\label{cor:ultbdd}
 Let $s\in \C$ with $\Re s>0$ satisfy the conditions of the theorem
and define $z\in \C$ by 
$ s = \frac{1-e^{-z}}{1+e^{-z}}.$ Then,
\begin{align*}
\ & \n \exp(-zL)\n_{\calL(L^p(\R^d,\gamma_{\a}),L^{q}(\R^d,\gamma_{\b}))} 
\\ & \qquad \le \frac{2^dC}{|1-e^{-2z}|^{\frac{d}{2}}}
\frac1{(1 - \frac{2}{\a p}+ \Re 
\frac{1+e^{-2z}}{1-e^{-2z}})^{\frac{d}{2}(1-\frac1p)}
(\frac{2}{\b q}-1 + \Re \frac{1+e^{-2z}}{1-e^{-2z}})^{\frac{d}{2}\frac1q}},
\end{align*}
where $C$ is the numerical constant in Theorem \ref{thm:main} (cf. Remark 
\ref{rem:C}).
\end{corollary}
\begin{proof}
Noting that $2/(1+e^{-z}) = 1+s$, we have
 \begin{align*}
  \n \exp(-zL)\n 
  & \le |1+s|^d  \n \exp(-s(P^2+Q^2))\n
  \\ & \le C|1+s|^d \frac1{|s|^{d/2}} 
  \frac1{|1 - \frac{2}{\a p}+ r_+(s)|^{\frac{d}{2}(1-\frac1p)}| 
  \frac{2}{\b q}-1 + r_+(s)|^{\frac{d}{2}\frac1q}}.
\end{align*}
The result follows from this by substituting 
$r_+(s) = \frac12 \Re(\frac1s+s) = \Re \frac{1+e^{-2z}}{1-e^{-2z}}$.
\end{proof}

\section{Restricted $L^p$-$L^2$ boundedness and Sobolev embedding}\label{sec:bbg}

As a first application of Theorem \ref{thm:main} we have the following 
`hyperboundedness' result for real times $t>0$:

\begin{corollary}\label{cor:bbg}
For $p \in [1,2]$ and $t>0$ set $\a_{p,t} := (1+e^{-2t})/p.$
\begin{enumerate}
 \item[\rm(1)] For all $t>0$ the operator $\exp(-tL)$ is bounded from 
$L^{1}(\R^d,\gamma_{\a_{1,t}})$ to $L^2(\R^d,\gamma)$, with norm 
$$ \n \exp(-tL)\n_{\calL(L^{1}(\R^d,\gamma_{\a_{1,t}}),L^2(\R^d,\gamma)} \lesssim_{d} 
(1-e^{-4t})^{-d/4}.$$       
 \item[\rm(2)] For all $p \in [1,2]$ and $t>0$ the operator $\exp(-tL)$ is bounded from 
$L^{p}(\R^d,\gamma_{\a_{p,t}})$ to $L^2(\R^d,\gamma)$, with norm 
$$ \n \exp(-tL)\n_{\calL(L^{p}(\R^d,\gamma_{\a_{p,t}}),L^2(\R^d,\gamma)} \lesssim_{d,p} 
t^{-\frac{d}{2}(\frac{1}{p}-\frac{1}{2})} \ \  \hbox{as $t\downarrow 0$}.$$
\end{enumerate}
\end{corollary}
\begin{proof}
Elementary algebra shows that with  $\a_{p,t}  = \frac2p(1+ \frac{2s}{1+s^2}) $
and $s =\frac{1-e^{-t}}{1+e^{-t}}$,
the criterion of Theorem \ref{thm:main} holds for all $t\ge 0$ (with equality in \eqref{eq:s}).
Both norm estimates follow from Corollary \ref{cor:ultbdd},  the fist by taking $p=1$, 
the second by noting that 
$\Re(\frac{1+e^{-2t}}{1-e^{-2t}}) \sim \frac{1}{t}$ for small values of $t$.
\end{proof}

A sharp version of this corollary is due to Bakry, Bolley and Gentil 
\cite[Section 4.2, Eq. (28)]{bbg}, who showed (for $p=1$) the hypercontractivity 
bound 
$$ \n \exp(-tL)\n_{\calL(L^1(\R^d,\gamma_{\a_{1,t}}),L^2(\R^d,\gamma))} \le 
(1-e^{-4t})^{-d/4}.$$
Their proof relies on entirely different techniques which seem not to generalise 
to complex time so easily.

The next corollary gives `ultraboundedness' of the operators $\exp(-zL)$ for 
arbitrary $\Re z>0$ from  
$L^{p}(\R^d,\gamma_{2/p})$ into $L^{2}(\R^d,\gamma)$:

\begin{corollary}\label{cor:ultra}
Let $p \in [1,2]$.
For all $z \in \C$ with $\Re z>0$ the operator $\exp(-zL)$ maps
$L^{p}(\R^d,\gamma_{2/p})$ into $L^{2}(\R^d,\gamma)$. As a consequence, the 
semigroup generated by $-L$ extends to
a strongly continuous
holomorphic semigroup of angle $\frac12\pi$ on $L^{p}(\R^d,\gamma_{2/p})$.
For each $\theta\in (0,\frac12\pi)$ this semigroup is uniformly bounded on the 
sector $\{z\in \C: \, z\not=0,\, |\arg(z)|<\theta\}$.
\end{corollary}

\begin{proof}
This follows from Corollary \ref{cor:ultbdd} upon realising that the assumptions 
of Theorem \ref{thm:main} are satisfied when $q=2$, $\b=1$ and $\a = \frac2{p}$, 
or 
$q=p$ and $\a = \b = \frac2{p}$.
\end{proof}

A notable consequence of Corollary \ref{cor:ultbdd} is the following (restricted) 
Sobolev embedding result. 
It is interesting because $(I+L)^{-1}$ maps $L^p(\R^d,\gamma)$ into $L^{2}(\R^d, 
\gamma)$ only when $p=2$ (i.e. no full Sobolev embedding theorem holds in the 
Ornstein-Uhlenbeck context).

\begin{corollary}[Restricted Sobolev embedding]
Let $p \in (\frac{2d}{d+2},2]$.
The resolvent $(I+L)^{-1}$ maps $L^p(\R^d,\gamma_{2/p})$ into $L^{2}(\R^d, 
\gamma)$.
\end{corollary}

\begin{proof}
Let $p \in (\frac{2d}{d+2},2]$ and fix $u\in L^p(\R^d,\gamma_{2/p}) \cap 
L^{2}(\R^d, \gamma)$.
Then
$$
\|\exp(-t(I+L))u\|_{L^{2}(\R^d, \gamma)} \lesssim_{d,p}
\|u\|_{L^p(\R^d,\gamma_{2/p})}
\exp(-t)t^{-\frac{d}{2}(\frac{1}{p}-\frac{1}{2})} \quad \forall t \geq 0,
$$
and thus $$\|(I+L)^{-1}u\|_{L^{2}(\R^d, \gamma)} \lesssim_{d,p} 
\|u\|_{L^p(\R^d,\gamma_{2/p})} \int_{0} ^{\infty} 
\exp(-t)t^{-\frac{d}{2}(\frac{1}{p}-\frac{1}{2})} \ud t 
\lesssim \|u\|_{L^p(\R^d,\gamma_{2/p})}, $$
since $p \in (\frac{2d}{d+2},2]$ implies 
$\frac{d}{2}(\frac{1}{p}-\frac{1}{2})<1$.
\end{proof}

\section{$L^p$-$L^q$ Boundedness}\label{sec:Epp}

We now turn to the classical setting of the spaces $L^p(\R^d,\gamma)$, where 
$\gamma$ is the standard Gaussian measure.
For $\a=\b=1$ and $s = x+iy$ the first positivity condition of Theorem \ref{thm:main} takes the form
\begin{align}\label{eq:pq1} 
1-\frac{2}{p} + r_+(s) > 0 &\ \Longleftrightarrow \ 1-\frac2p + \frac12\Bigl(\frac{x}{x^2+y^2} +x\Bigr) > 0
\end{align}
whereas condition \eqref{eq:s} is seen to be equivalent to the condition
\begin{align}\label{eq:pq2}
(p-q)(x +\frac{x}{x^2+y^2}) +pq (\frac{x^2}{x^2+y^2}  -1) +2p + 2q-4 \ge  0.
\end{align}
Let us also observe that if these two conditions hold, together they enforce the second positivity
condition $\frac{2}{q} -1 + r_+(s) > 0$; this is apparent from the representation in \eqref{eq:s}.

As a warm up for the general case, let us first consider real times $t\in (0,1)$ in the $z$-plane, 
which correspond to the values $s = x\in (0,1)$ in the $s$-plane. The conditions \eqref{eq:pq1} and \eqref{eq:pq2}
then reduce to $$(1-\frac2p)x +\frac12(x^2+1) >0$$ and $$ 
(p-q)(x +\frac1{x}) +2p + 2q-4 \ge  0,$$ respectively.
The first condition is automatic. Substituting $x = \frac{1-e^{-t}}{1+e^{-t}}$ in the second and solving for $e^{-t}$, 
assuming $p\le q$ we 
find that it is equivalent to the condition $$e^{-2t}\le \frac{p-1}{q-1}.$$
Thus we recover the boundedness part of Nelson's celebrated hypercontractivity 
result \cite{Nel}. 

Turning to complex time, with some additional effort we also recover the following result due to Weissler \cite{Wei} (see also Epperson 
\cite{Epp} for further refinements), essentially as a Corollary of Theorem \ref{thm:main}.

\begin{theorem}[$L^p$-$L^q$-boundedness of $\exp(-zL)$]\label{thm:LpLq} 
 Let $1< p\le q< \infty$.  
If $z\in \C$ satisfies $\Re z>0$, 
\begin{align}\label{eq:pqz0}
|e^{-z}|^2 < p/q                                   
\end{align}
 and 
\begin{equation}\label{eq:pqz} 
(q-1)|e^{-z}|^4 + (2-p-q) (\Re e^{-z})^2 - (2-p-q+pq)(\Im e^{-z})^2 + p-1 > 0,
\end{equation}
then the operator
$\exp(-zL)$ maps $L^p(\R^d,\gamma)$ into $L^{q}(\R^d,\gamma)$.
\end{theorem}

Before turning to the proof we make a couple of preliminary observations.
By a simple argument involving quadratic forms (see \cite[page 3]{Epp}), the 
conditions \eqref{eq:pqz0}
and \eqref{eq:pqz} 
taken together are equivalent to the single condition 
\begin{align}\label{eq:epp} (\Im (we^{-z}))^2 + (q-1) 
(\Re (we^{-z}))^2  < (\Im w)^2 + (p-1) (\Re w)^2 \quad \forall w\in\C.
\end{align}
Let us denote the set of all $z\in \C$, $\Re z\ge 0$, 
for which \eqref{eq:epp} holds by $E_{p,q}$. 
The following two facts hold:

\begin{facts} \ 

\begin{itemize}
 \item [$\bullet$] $E_{p,p} = E_p$.
 \item [$\bullet$] $E_{p,q} \subseteq E_p$ and $E_{p,q} \subseteq E_q$.
\end{itemize}
\end{facts}
The first is implicit in \cite{Epp, GMMST}, can be proved by elementary means, and is taken for granted.
The second is an immediate consequence of the assumption $p\le q$.

Let us now start with the proof of Theorem \ref{thm:LpLq}. 
It is useful to dispose of the positivity condition \eqref{eq:pq1} in the form of a lemma;
see also Figure 2.
 
 \begin{center}
 \begin{figure}[ht]
  \includegraphics[scale=0.5]{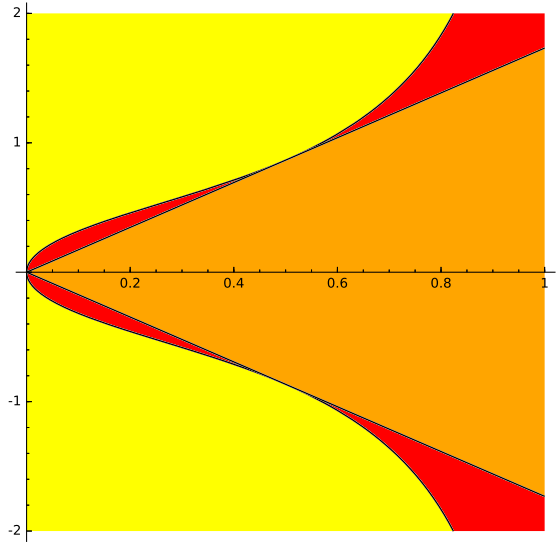}
 
  \caption{
 The region $R_p := \{s\in \C: \ 1-\frac{2}{p} + r_+(s) > 0\}$ (red/orange) 
 and the sector $\Sigma_p$ (orange), both for $p = 4/3$. Lemma \ref{lem:sect}
 implies that $\Sigma_p$ is indeed contained in $R_p$.}
 \end{figure}
 \end{center}


Let $z\in \C$ satisfy $\Re z>0$. 
By the remarks at the end of Section 3, $z$ belongs to $E_p$ if and only if  
$s = \frac{1-e^{-z}}{1+e^{-z}}$ belongs to $\Sigma_{\phi_p}\setminus\{1\}$. 

\begin{lemma}\label{lem:sect} Every $s\in \Sigma_{\phi_p}$ satisfies the positivity condition \eqref{eq:pq1}.
\end{lemma}
\begin{proof}
Writing $s = x+iy$, we then have
 $$ \frac{x^2}{x^2+y^2} = \cos^2\theta_p > (1-\frac2p)^2,$$
where the angle $\theta_p$ is given by \eqref{eq:theta-p}. 
 To see that this implies \eqref{eq:pq1}, note that
\begin{equation}\label{eq:quadr}
 \begin{aligned}
  1-\frac2p + \frac12\Bigl(\frac{x}{x^2+y^2} +x\Bigr) 
   > 2(1-\frac2{p})x +  (1-\frac2p)^2 + x^2 = \Bigl(x+ (1-\frac2p)\Bigr)^2 
   \end{aligned}
\end{equation}
and the latter is trivially true.
\end{proof}

\begin{proof}[Proof of Theorem \ref{thm:LpLq}]
Fix $\Re z>0$ and set $s:= \frac{1-e^{-z}}{1+e^{-z}}$.
We show that the assumptions of the theorem imply the conditions of Theorem 
\ref{thm:main},
 so that $\exp(-s(P^2 + Q^2))$ maps $L^p(\R^d,\gamma)$ into 
$L^{q}(\R^d,\gamma)$.
In combination with Theorem \ref{thm:exptL}, this gives the result.
 
We begin by checking the condition \eqref{eq:pq1}. For this, the second fact tells us that  
there is no loss of generality in assuming that $q=p$. In that situation, the first fact
tells us that $z$ belongs to $E_p$. But then Lemma \ref{lem:sect} gives us the desired result.
 
It remains to check \eqref{eq:pq2}.
Multiplying both sides with $x^2+y^2$, this can be 
rewritten as
\begin{align}\label{eq:sa} (p-q)x(1+x^2+y^2) + pq x^2 - (pq-2p-2q+4)(x^2+y^2) 
\ge 0.
\end{align}
The proof of the theorem is completed by showing that \eqref{eq:pqz} implies \eqref{eq:sa}.

Towards this end, we rewrite \eqref{eq:pqz} in a similar way.  
Setting $e^{-z} = \frac{1-s}{1+s}$ with $s = x+iy$, and using that 
$$ \Re \frac{1-x-iy}{1+x+iy} = \frac{1-(x^2+y^2)}{(1+x)^2+y^2}, \quad 
\Im \frac{1-x-iy}{1+x+iy} = -\frac{2y}{(1+x)^2+y^2},
$$
\eqref{eq:pqz} takes the form 
\begin{align*}  \ & (q-1)((1-(x^2+y^2))^2+4y^2)^2 
+ (2-p-q)(1-(x^2+y^2))^2((1+x)^2+y^2)^2 \\ 
& \qquad\qquad\qquad\qquad
- (2-p-q+pq)4y^2((1+x)^2+y^2)^2+(p-1)((1+x)^2+y^2)^4 >0.
\end{align*}
This factors as 
\begin{equation}\label{eq:z1} \big[4((1+x)^2+y^2)^2\big] 
\times \big[(p-q)x(1+x^2 + y^2)+ (2p+2q-4)x^2   - (pq - 2p - 2q +4)y^2\big].
\end{equation}
Quite miraculously, the second term in straight brackets precisely equals the term in \eqref{eq:sa}.
Since $4((1+x)^2+y^2)^2 >0$ it follows that  \eqref{eq:z1} (and hence \eqref{eq:pqz})
implies \eqref{eq:sa}
(and hence \eqref{eq:s}). 
\end{proof}

It is shown in \cite{Epp} (see also \cite{Jans}) that the operator 
$\exp(-zL)$ is bounded from $L^p(\R^d,\gamma)$ to $L^q(\R^d,\gamma)$ if and only 
if $z\in \overline{E_p}$, and then the operators $\exp(-zL)$ are in fact 
contractions. 
Our proof does not recover 
the contractivity of $\exp(-zL)$. Nevertheless it is
remarkable that the boundedness part does follow from our method, which just uses 
\eqref{eq:exppq}, elementary calculus, the Schur test, and some algebraic manipulations.

For $p=q$, Theorem \ref{thm:LpLq} combined with the fact that $E_{p,p} = E_p$   
contains as a special case that, for a given $z\in \C$ 
with $\Re z>0$, the operator $\exp(-zL)$ is bounded on $L^p(\R^d,\gamma)$ if $z$ belongs 
to ${E_p}$. 
A more direct - and more transparent - proof 
of this fact may be obtained as a consequence of the following theorem.

\begin{theorem}\label{thm:Epp} For all $1<p<\infty$ and $s\in 
\Sigma_{\theta_p}$ 
the operator $\exp(-s(P^2 + Q^2))$ is bounded on $L^p(\R^d,\gamma)$.
\end{theorem}
As we explained in Section \ref{sec:OU}, this result translates into Epperson's 
result that the semigroup $\exp(-tL)$ on $L^p(\R^d,\gamma)$ can be analytically extended to 
to $E_p$. 

\begin{proof}
Lemma \ref{lem:sect} shows that \eqref{eq:pq1} holds. 
Since $q=p$, \eqref{eq:pq2} reduces to the  condition
 $$ p^{2}(\frac{x^{2}}{x^{2}+y^{2}}-1) + 4p - 4 > 0,
 $$ which is equivalent to saying that $s\in \Sigma_{\theta_p}$.
\end{proof}

\begin{remark}
More generally, for an arbitrary pair $(\a,p) 
\in [1,\infty)\times [1,\infty)$ satisfying $\a p > 2$, by the same method 
we obtain that $\exp(-zL)$ is bounded 
on $L^p(\R^d,\gamma_{\a})$ if $s = \frac{1-e^{-z}}{1+e^{-z}}$ satisfies
$$
\frac{\Re s}{|s|} > 1-\frac{2}{\a p}.
$$
This corresponds to the sector of angle 
$\theta_{\a,p}=\arccos(1-\frac{2}{\a p})$ 
in the $s$-plane.
In the $z$-plane, this corresponds to the Epperson region $E_{\alpha p}$.
\end{remark}

{\em Acknowledgment --} We thank Emiel Lorist for generating the figures.

\end{document}